\documentclass[twoside]{amsart}
\usepackage{latexsym}
\usepackage{amssymb,amsmath,amsopn}
\usepackage[dvips]{graphicx}   
\usepackage{color,epsfig}      
\usepackage{tikz} 
\usepackage{subcaption} 
\usepackage{eurosym}
\newtheorem{thm}{Theorem}

\newtheorem{lem}{Lemma}
\newtheorem{cor}{Corollary}
\newtheorem{prop}{Proposition}
\theoremstyle{definition}
\newtheorem{defn}{Definition}
\newtheorem{rem}{Remark}

\newtheorem{conj}{Conjecture}

\renewcommand{\Re}{\mathbb R}
\newcommand{\Red}{\Re^d}
\newcommand{\Z}{\mathbb Z}
\newcommand{\N}{\mathbb N}

\newcommand{\Sph}{\mathbb{S}}

\newcommand{\bd}{\partial}

\DeclareMathOperator{\inter}{int}

\DeclareMathOperator{\conv}{conv}
\DeclareMathOperator{\card}{card}

\DeclareMathOperator{\area}{area}

\DeclareMathOperator{\cl}{cl}

\DeclareMathOperator{\vol}{vol}

\begin{document}

\title{Volume of the Minkowski sums of star-shaped sets}

\author[M. Fradelizi]{Matthieu Fradelizi}

\author[Z. L\'angi]{Zsolt L\'angi}

\author[A. Zvavitch]{Artem Zvavitch}

\thanks{The first author is supported in part by the Agence Nationale de la Recherche, projet ASPAG - ANR-17-CE40-0017; the second author is partially supported by the National Research, Development and Innovation Office, NKFI, K-119670, the J\'anos Bolyai Research Scholarship of the Hungarian Academy of Sciences, and grants BME FIKP-V\'IZ and \'UNKP-19-4 New National Excellence Program by the Ministry of Innovation and Technology; the third  author is supported in part by the U.S. National Science Foundation Grant DMS-1101636 and the B\'ezout Labex funded by ANR, reference ANR-10-LABX-58.}

\address{LAMA, Univ Gustave Eiffel, Univ Paris Est Creteil, CNRS, F-77447 Marne-la-Vallée, France}\email{matthieu.fradelizi@u-pem.fr }

\address{Morphodynamics Research Group and Department of Geometry, Budapest University of Technology, Egry J\'ozsef utca 1., Budapest 1111, Hungary}  \email{zlangi@math.bme.hu}

\address{Department of Mathematical Sciences, Kent State University,
Kent, OH 44242, USA}  \email{zvavitch@math.kent.edu}

\date{\today}



\begin{abstract}
For a compact set $A \subset \Re^d$ and an integer $k\ge1$, let us denote by
$$ A[k] = \left\{a_1+\cdots +a_k: a_1, \ldots, a_k\in A\right\}=\sum_{i=1}^k A$$
the Minkowski sum of $k$ copies of $A$.
A theorem of Shapley, Folkmann and Starr (1969) states that $\frac{1}{k}A[k]$ converges to the convex hull of $A$ 
in Hausdorff distance as $k$ tends to infinity. Bobkov, Madiman and Wang (2011) conjectured that 
the volume of $\frac{1}{k}A[k]$ is non-decreasing in $k$,
or in other words, in terms of the volume deficit between the convex hull of $A$ and $\frac{1}{k}A[k]$, this convergence is monotone.
It was proved by   Fradelizi,   Madiman,   Marsiglietti and  Zvavitch (2016) that this conjecture holds true if $d=1$ but fails for any $d \geq 12$. 
In this paper we show that the conjecture is true for any star-shaped set $A \subset \Re^d$  for $d=2$ and $d=3$ and also for arbitrary dimensions $d \ge 4$ under the condition $k \ge (d-1)(d-2)$. In addition, we investigate the conjecture for connected sets and present a counterexample to a generalization of the conjecture to the Minkowski sum of possibly distinct sets in $\Red$, for any $d \geq 7$.
\end{abstract}

\maketitle

\section{Introduction}

\parskip=4pt

The Minkowski sum of two sets $K, L \subset \Re^d$  is defined as $K+L=\{x+y: x \in K, y \in L\}$, where, for brevity, we set $A[k] = \sum_{i=1}^k A$, for any $k \in {\mathbb N}$ and any compact set $A \subset \Red$.
Since Minkowski sum preserves the convexity of the summands and the set $\frac{1}{k}A[k]$ consists in some particular convex combinations of elements of $A$, the containment $\frac{1}{k}A[k] \subseteq \conv A$, and, for the special case of convex sets, the equality $\frac{1}{k}A[k] =\conv A$ trivially holds; here $\conv A$ denotes the convex hull of $A$. 
These observations suggest that for any compact set $A$, the set $\frac{1}{k}A[k]$ looks ``more convex'' for larger values of $k$.
This intuition was formalized by Starr \cite{St1, St2}, crediting also Shapley and Folkman, and independently by Emerson and Greenleaf \cite{EG}, by proving that the set $\frac{1}{k}A[k]$ approaches $\conv A$ in Hausdorff distance as $k$ approaches infinity and by giving bounds on the speed of this convergence (we refer to \cite{FMMZ2} for more discussion of this fact).

A further step in the investigation of the sequence $\left\{ \frac{1}{k}A[k] \right\}$ is to examine the monotonicity of this convergence. Whereas this sequence is clearly not monotonous in terms of containment, the main object of this paper is the following conjecture of  Bobkov, Madiman, Wang  \cite{BMW}, relating the volumes of the elements of the sequence, and in which $\vol(K)$ denotes the Lebesgue measure (volume) of the measurable set $K \subset \Re^d$.

\begin{conj}[Bobkov-Madiman-Wang]\label{conj:BMW}
Let $A$ be a compact set in $\Re^d$ for some $d \in \mathbb{N}$. Then the sequence 
$$\left\{ \vol \left( \frac{1}{k}A[k] \right) \right\}_{k \geq 1}$$ 
is non-decreasing in $k$.
\end{conj}
Equivalently, this conjecture asks whether for any integer $k \ge1$ and compact set $A \subset \Re^d$, the following inequality holds
\begin{equation}\label{eq:BMW}
\vol \left( \frac{1}{k} A[k] \right) \leq \vol \left( \frac{1}{k+1} A[k+1] \right).
\end{equation}
This inequality trivially holds for any compact set $A$ if $k=1$ since $A\subset \frac{1}{2} A[2]$.
In the same way, it is easy to find monotone subsequences of the sequence $\{\vol(\frac{1}{k} A[k])\}_{k\ge 1}$ by the same argument;
one such example is $\{\vol(\frac{1}{2^m} A[2^m])\}_{m\ge 0}$.
On the other hand, even the first nontrivial case; that is, the inequality $\vol \left( \frac{1}{2}A[2] \right) \leq \vol \left( \frac{1}{3}A[3] \right)$ seems to require new methods to approach.
Conjecture~\ref{conj:BMW} was partially resolved in \cite{FMMZ1, FMMZ2}, where, following the approach of \cite{GMR}, the authors proved it for any $1$-dimensional compact set $A$, but constructed counterexamples in $\Red$ for any $d \geq 12$. 
More precisely, they showed that 
for every $k \geq 2$, there is $d_k \in \mathbb{N}$ such that for every $d \geq d_k$ there is a compact set $A \subset \Re^d$ such that
$\vol\left( \frac{1}{k} A[k] \right) > \vol\left( \frac{1}{k+1} A[k+1] \right)$. In particular, one has $d_2 = 12$, whence Conjecture~\ref{conj:BMW} fails for $\Re^d$ if $d \geq 12$.

Our goal is to find additional conditions on $A$ and $k$ under which the statement in Conjecture~\ref{conj:BMW}, or more precisely when the inequality (\ref{eq:BMW}) is satisfied.

In the paper, for any set $A \subset \Red$ we denote by $\dim A$ the dimension of the smallest affine subspace containing $A$, and for any $p, q \in \Re^d$, we denote the closed segment with endpoints $p, q$ by $[p,q]$.
To state our main result, let us recall the following well-known concept.

\begin{defn}
A nonempty set $S \subset \Re^d$ is called \emph{star-shaped} with respect to a point $p$ if for any $q \in S$, we have $[p,q] \subseteq S$. 
\end{defn}
Our main result is the following.
\begin{thm}\label{thm:main}
Let $d \ge 2$ and $k \geq \max\{2, (d-1)(d-2) \}$ be  integers. Then for any compact, star-shaped set $S \subset \Re^d$ we have
\[
\vol\left( \frac{1}{k+1} S[k+1] \right) \geq \vol\left( \frac{1}{k} S[k] \right),
\]
with equality if only if    $\dim(S) < d$ or $\frac{1}{k} S[k] = \conv (S)$.
\end{thm}

We notice that Theorem~\ref{thm:main} establishes Conjecture~\ref{conj:BMW} for star-shaped compact sets in dimensions 2 and 3.
It is worth  to remark  that the compact sets $A$ constructed in \cite{FMMZ2} as counterexamples to Conjecture~\ref{conj:BMW} are star-shaped, which makes Theorem~\ref{thm:main} fairly unexpected.

We prove Theorem \ref{thm:main} in Section~\ref{sec:proof}. In Section~\ref{sec3} we adapt our techniques to investigate connected sets.
Our main result in this section is summarized in Theorem~\ref{thm:convexholes}. Finally, in Section~\ref{sec:remarks} we collect some additional remarks and questions, and, in particular, we construct low dimensional counterexamples to a generalization of Conjecture~\ref{conj:BMW}, which also appeared in \cite{BMW}.

\noindent {\bf Acknowledgment:}  The authors  are grateful to the anonymous referee for a careful reading of the paper and constructive comments and corrections.

\section{Conjecture~\ref{conj:BMW} for star-shaped sets: the proof of Theorem~\ref{thm:main}}\label{sec:proof}

We start this section with a couple of Lemmata which are needed for the proof.
Throughout this section, we denote $X_d(t)=\{(x_1, \dots, x_d) \in \N^d: x_1+\ldots + x_d = t \}$ and $N_d(t)=\card X_d(t) $ to be the number of elements of $X_d(t)$.  Here and in the rest of the paper we will denote by $\N$ the set of non-negative integers.

\begin{lem}\label{lem:integerpoints}
For any integer $t \geq 1$, and $d \geq 2$, we have $N_d(t) = \binom{t+d-1}{d-1}$.
\end{lem}

\begin{proof}
If $d=2$, then, clearly, $N_2(t) = t+1 = \binom{t+2-1}{1}$.
On the other hand, by induction, we have
\[
N_d(t) = \sum_{s=0}^t N_{d-1}(s) = \sum_{s=0}^t \binom{s+d-2}{d-2} = \binom{t+d-1}{d-1}.
\]
\end{proof}
\begin{lem}\label{lem:main}
Let  $d\ge2$ and $o$ be the origin of $\Re^d$,  $(p_1,\ldots,p_d)$ be a basis of $\Re^d$, and, let $B= \bigcup_{i=1}^d [o,p_i]$. Consider a compact set $M\subset\Re^d$  such that $B[k] \subseteq M \subseteq k \conv (B)$  for some $k \geq \max \{ 2, (d-1)(d-2) \}$, then
\begin{equation}\label{eq:addB}
\vol\left(\frac{1}{k+1}(M+B) \right) \geq \vol \left( \frac{1}{k} M \right),
\end{equation}
where, equality holds if and only if $M = k \conv (B)$. Furthermore, if $\vol \left( \frac{1}{k} M \right) \geq \vol\left(\frac{1}{k+1}(M+B) \right) - \delta$ for some $\delta \geq 0$, then $\vol (M) \geq \vol\left( k \conv (B) \right) - C(d,k) \delta$,  where the constant $C(d,k)=k^d (1- \frac{k^d}{(k-d+2) (k+1)^{d-1}})^{-1}$ depends  only on $d$ and $k$.
\end{lem}

\begin{proof}
Since the inequality  (\ref{eq:addB})  is independent of  a non-degenerate linear transformation applied to $B$ and $M$ simultaneously, we may assume that $(p_1, \ldots, p_d)$ is the canonical basis of $\Re^d$. Let 
$$
V(t)=\vol\{(x_1, \dots, x_d) \in [0,1]^d: x_1+\ldots + x_d \leq t \}.
$$
Let $C_i=i+[0,1]^d$, $i \in \Z^d$ be the unit cube cells of the lattice $\Z^d$, 
and set $\mu_i = \vol(C_i \cap M)$, and $\lambda_i = \vol (C_i \cap (M + B) )$.

Note that for any $i \in X_d(t)$, $\vol(C_i \cap k \conv(B))$ is independent of $i$, namely it is equal to $1$, if $t \leq k-d$, and to $V(k-t)$ if $t=k-d+1, \ldots, k-1$.
A similar statement holds for $\vol(C_i \cap (k+1) \conv(B))$. The number of unit cells contained in $k \conv(B)$ is equal to the number of the solutions of the inequality $x_1 + x_2  + \ldots + x_d \leq k$, where each variable is a positive integer, and thus, it is $\binom{k}{d}$.
 Hence, if $Y_d(k)$ denotes the union of these cells, then we have that
\begin{equation}\label{eq:containedcells}
\vol (Y_d(k)) = k^{\underline{d}} V,
\end{equation}
where $k^{\underline{d}}=k(k-1)\ldots (k-d+1)$, and $V = \vol(\conv B) = \frac{1}{d!}$. Thus,
\begin{equation}\label{eq:kconvB}
\vol(M) = k^{\underline{d}} V+\sum_{t=k-d+1}^{k-1} \sum_{i \in  X_d(t) } \mu_i,
\end{equation}
and
\[
\vol(M+B) = (k+1)^{\underline{d}}V  + \sum_{t=k-d+2}^{k} \sum_{i \in  X_d(t) } \lambda_i.
\]

In the following step, we give a lower bound on the $\lambda_i$'s depending on the values of the $\mu_i$'s.
We say that $i \in X_d(t)$ and $i' \in X_d(t+1)$ are \emph{adjacent} if the corresponding cells $C_i$ and $C_{i'}$ have a common facet, or in other words,
if $i'-i$ coincides with one of the standard basis vectors $p_j$. In this case we write $ii' \in I$.
Let $i \in X_d(t)$, and let $S = M \cap C_i$. Then, for every $j=1,2,\ldots,d$, $S+p_j \subset (M+B) \cap C_{i'}$ with $i'=i+p_j$.
Thus, for any $i \in X_d(t+1)$,
\begin{equation}\label{eq:estimate1}
\lambda_i \geq \max \{ \mu_{i'}: i'\in X_d(t) \hbox{ is adjacent to } i\}.
\end{equation}
Note that the right-hand side of this inequality is not less than any convex combination of the corresponding $\mu_{i'}$s.
Using a suitable convex combination for each $i \in X_d(t+1)$, we show that this inequality implies that
\begin{equation}\label{eq:estimate_main}
\sum_{i \in X_d(t+1)} \lambda_i \geq \frac{t+d}{t+1} \sum_{i \in  X_d(t)} \mu_i.
\end{equation}

\begin{figure}[ht]
\begin{center}
\includegraphics[width=.45\textwidth]{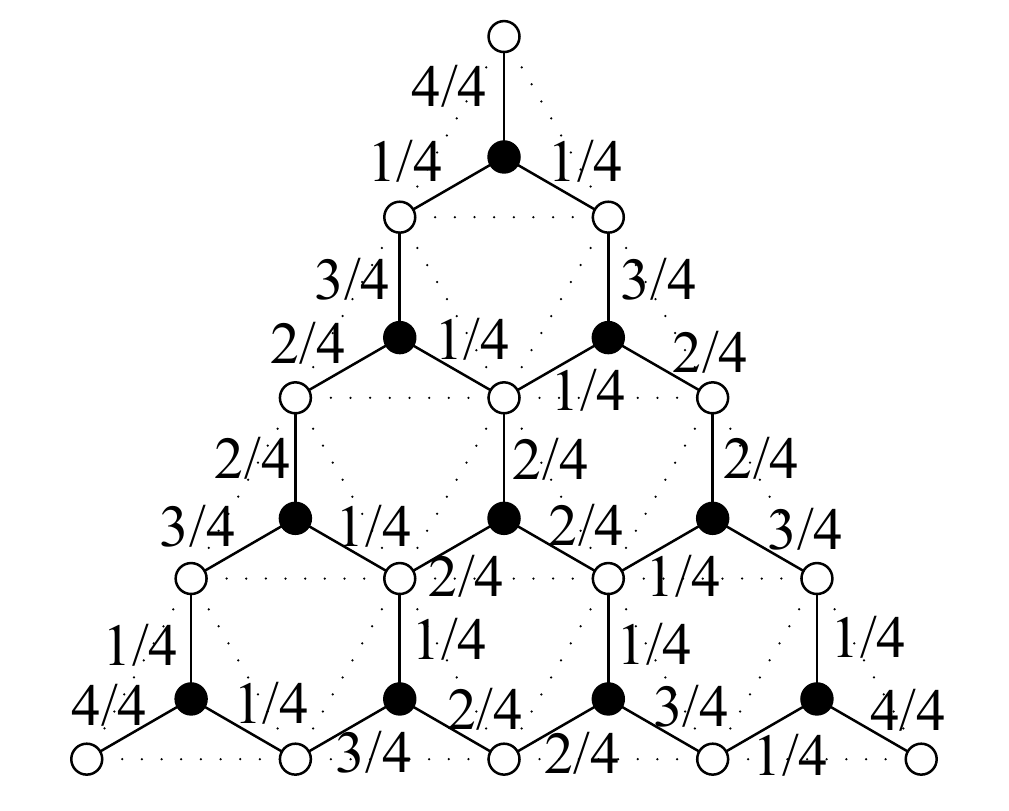}
\caption{Illustration on choosing the weights if $d=3$ and $t=3$. The black and empty dots represent the elements of the set $X_3(3)$ and $X_3(4)$, respectively. Dots illustrating adjacent indices are connected by a segment. The weight assigned to the segment connecting the dots representing $i$ and $i'$ is equal to $\alpha_{ii'}$.}
\label{fig:layers}
\end{center}
\end{figure}

Consider some $i=(i_1,i_2,\ldots,i_d) \in X_d(t+1)$. Then the indices in $X_d(t)$ adjacent to $i$ are all of the form $i-p_j$ for some  $j=1,2,\ldots,d$. Furthermore,
$i-p_j$ is adjacent to $i$ iff $i_j \geq 1$, or in other words, iff $i_j \neq 0$. Now, for any $i' \in X_d(t)$ adjacent to $i$ we set $\alpha_{ii'} = \frac{i_j}{t+1}$,
where $i-i'=p_j$ (cf. Figure~\ref{fig:layers}).  Then, since $i \in X_d(t+1)$, we clearly have $1 = \sum_{j=1}^d \frac{i_j}{t+1} = \sum_{i' \in X_d(t), ii' \in I} \alpha_{ii'}$.
Thus, by (\ref{eq:estimate1}), we have
\begin{equation}\label{eq:estimate2}
\lambda_i \geq \sum_{i' \in X_d(t), ii' \in I} \alpha_{ii'} \mu_{i'}
\end{equation}
for all $i \in X_d(t+1)$. 
Now, let $i' \in X_d(t)$, and $i'=(i'_1,i'_2,\ldots,i'_d)$. Then the indices in $X_d(t+1)$ adjacent to $i'$ are exactly those of the form $i'+p_j$ for some $i=1,2,\ldots,d$.
Hence,
\begin{equation}\label{eq:estimate3}
\sum_{i \in X_d(t+1), ii' \in I} \alpha_{ii'} = \sum_{j=1}^d \frac{i'_j+1}{t+1} = \frac{t+d}{t+1}.
\end{equation}
Finally, by (\ref{eq:estimate2}) and (\ref{eq:estimate3})
\begin{align*}
\sum_{i \in X_d(t+1)} \lambda_i \geq \sum_{i \in X_d(t+1)} \sum_{i' \in X_d(t), ii' \in I} \alpha_{ii'} \mu_{i'} &= \sum_{i' \in X_d(t)} \left( \sum_{i \in X_d(t+1), ii' \in I} \alpha_{ii'} \right) \mu_{i'} \\
& = \frac{t+d}{t+1} \sum_{i' \in X_d(t)} \mu_{i'}.
\end{align*}

Using this inequality and the assumption that $B[k] \subseteq M \subseteq k \conv (B)$, we obtain
\[
\vol(M+B) \geq (k+1)^{\underline{d}}V + \sum_{t=k-d+1}^{k-1} \frac{t+d}{t+1} \sum_{i \in X_d(t)} \mu_i .
\]
Note that the sequence $\left\{ \frac{t+d}{t+1} \right\}$, where  $t=0,1,2,\ldots$, is strictly decreasing.
Hence, using the fact that if $i \in X_d(t)$, then $\mu_i \leq V(k-t)$,  one has, for $k-d+1\le t\le k-1$,
\begin{align*}
\frac{t+d}{t+1} \sum_{i \in X_d(t)} \mu_i &\ge \frac{k+1}{k-d+2} \sum_{i \in X_d(t)}\mu_i+\left(\frac{t+d}{t+1} -\frac{k+1}{k-d+2}\right) V(k-t) N_d(t)\\
&\ge \frac{k+1}{k-d+2}\left(\sum_{i \in X_d(t)}\mu_i-V(k-t) N_d(t)\right)+\frac{t+d}{t+1}V(k-t) N_d(t).
\end{align*}
Hence
\begin{multline}\label{eq:important}
\vol(M+B) \geq (k+1)^{\underline{d}}V +\frac{k+1}{k-d+2}\sum_{t=k-d+1}^{k-1}\left(\sum_{i \in X_d(t)}\mu_i-V(k-t) N_d(t)\right)\\
+ \sum_{t=k-d+1}^{k-1}\frac{t+d}{t+1}V(k-t) N_d(t).
\end{multline}
Observe that  $\sum_{t=k-d+1}^{k-1} V(k-t) N_d(t)= (k^d - k^{\underline{d}})V$, since it is the volume of the part of $k \conv (B)$ belonging to the cells that are not contained in  $k \conv (B)$, and the equality follows  by (\ref{eq:containedcells}). Similarly, since 
 \[
 \frac{t+d}{t+1}N_d(t) = \frac{t+d}{t+1} \binom{t+d-1}{d-1} =\binom{t+d}{d-1}=N_d(t+1)
\]
we deduce that
\[
\sum_{t=k-d+1}^{k-1} \frac{t+d}{t+1} V(k-t) N_d(t) 
= \sum_{t'=k-d+2}^{k} V(k+1-t') N_d(t')
 =((k+1)^d - (k+1)^{\underline{d}})V,
\]
since it is the volume of the part of $(k+1) \conv (B)$ belonging to cells that are not contained in  $(k+1) \conv (B)$. 
Substituting these into (\ref{eq:important}) and using (\ref{eq:kconvB}), we obtain
\begin{align*}
\vol(M+B)  &\geq (k+1)^{\underline{d}V  +  \frac{k+1}{k-d+2} \left( \vol(M) - k^dV \right) 
 +   ((k+1)^d - (k+1)^{\underline{d}})V} \\
 &\ge \frac{k+1}{k-d+2} \vol(M) +  \left( (k+1)^d - \frac{k+1}{k-d+2} k^d \right) V.
\end{align*}
Thus,
\begin{multline}\label{eq:lemma_final}
\vol\left( \frac{1}{k+1}(M+B) \right) \geq \frac{k^d}{(k-d+2) (k+1)^{d-1}} \vol\left( \frac{1}{k} M \right) \\+ \left(1 -  \frac{k^d}{(k-d+2) (k+1)^{d-1}} \right) V.
\end{multline}
Since $\vol\left( \frac{1}{k} M \right) \leq V$,  to prove the first inequality of the lemma,  it is sufficient to show that the right-hand side of (\ref{eq:lemma_final}) is a convex combination of the volumes, namely that the second coefficient is nonnegative. This is clear if $d=2$, while for $d \geq 3$ using
 the Binomial Theorem, one has
\begin{align*}
(k-d+2) (k+1)^{d-1} - k^d &>  (k-d+2) \left( k^{d-1} + (d-1) k^{d-2}\right) - k^d \\ &= k^{d-1} - (d-1)(d-2) k^{d-2},
\end{align*}
which is nonnegative for $k \geq (d-1)(d-2)$.

Now we prove the equality case. By (\ref{eq:lemma_final}), equality in the lemma implies that $ \vol\left( \frac{1}{k} M \right) = V$, or equivalently,
$\vol(k \conv (B) \setminus M) = 0$. Note that since 
$$\vol(k \conv(B)) > 0,$$
 its interior is not empty. Thus, $k \conv (B)$ is equal to the closure of its interior.
On the other hand, $\vol(k \conv (B) \setminus M) = 0$ implies that $\inter (k \conv B) \subset M$, but as $M$ is compact, $M=k \conv B$ follows.

Finally, if $\vol\left( \frac{1}{k+1}(M+B)\right) - \delta \leq \vol\left( \frac{1}{k} M \right)$, then in the same way  (\ref{eq:lemma_final}) yields the inequality
$\vol(M) \geq \vol(k \conv (B)) - C(d,k) \delta$, with 
\begin{equation}\label{eq:ckd}
C(d,k)=\frac{k^d}{1- \frac{k^d}{(k-d+2) (k+1)^{d-1}}}.
\end{equation}
\end{proof}


\begin{proof}[Proof of Theorem~\ref{thm:main}]
Without loss of generality, we may assume that $S$ is star-shaped with respect to the origin.
Let $\varepsilon > 0$ be an arbitrary positive number. By Carath\'eodory's theorem, we may choose a finite point set $A_0 \subset S$ such that $\vol(\conv(S)) - \varepsilon \leq \vol(\conv(A_0))$, and without loss of generality, we may assume that the points of $A_0$ are in convex position.
Clearly, the star-shaped set $A = \bigcup_{a \in A_0} [o,a]$ is a subset of $S$, satisfying $\vol(\conv(S)) - \varepsilon \leq \vol(\conv(A))$.
Consider a simplicial decomposition $\mathcal{F}$ of the boundary of $\conv (A)$ such that all vertices of $\mathcal{F}$ are vertices of $\conv(A)$.
Let the $(d-1)$-dimensional faces of $\mathcal{F}$ be $F_1, F_2, \ldots, F_m$, and for $j=1,2,\ldots, m$, let $B_j = \bigcup_{t=1}^d [o,p_t^j]$, where
$p_1^j, p_2^j, \ldots, p_d^j$ are the vertices of $F_j$. Then $B_j \subseteq S$ for all values of $j$, the sets $\conv (B_j)$ are mutually non-overlapping,
and $\conv(A) = \bigcup_{j=1}^m \conv(B_j)$.
Finally, let $M_j = S[k] \cap (k \conv (B_j))$. Then, since $B_j \subseteq S$, we have $B_j[k] \subseteq M_j \subseteq (k \conv(B_j))$.
Thus, Lemma~\ref{lem:main} implies that $\vol\left(\frac{1}{k+1}(M_j+B_j) \right) \geq \vol \left( \frac{1}{k} M_j \right)$. 
Thus, we have
\begin{equation*}
\begin{split}
\vol\left(\frac{S[k]}{k} \cap \conv(A)\right) =& \sum_{j=1}^m\vol\left(\frac{S[k]}{k} \cap \conv(B_j)\right)=
\sum_{j=1}^m \vol \left( \frac{M_j}{k} \right) \\
\leq & \sum_{j=1}^m \vol \left( \frac{M_j+B_j}{k+1} \right)\leq \vol\left(\frac{S[k+1]}{k+1}\right) .
\end{split}
\end{equation*}
On the other hand, since $0 \leq \vol(\conv(S))-\vol(\conv (A)) \leq \varepsilon$, we have 
\[
0 < \vol\left( \frac{S[k]}{k}\right)-\vol\left(\conv(A) \right) \leq \varepsilon,
\]implying that
 \begin{equation}\label{eq:equality}
\vol\left( \frac{S[k]}{k}  \right) - \varepsilon \leq \sum_{j=1}^m \vol \left( \frac{M_j}{k} \right) 
\leq  \sum_{j=1}^m \vol \left( \frac{M_j+B_j}{k+1} \right)\leq \vol\left(\frac{S[k+1]}{k+1}\right).
\end{equation}
This inequality is satisfied for all positive $\varepsilon$, and thus, the inequality part of Theorem~\ref{thm:main} holds.

Now, assume that 
\[
\vol\left( \frac{S[k]}{k}  \right) = \vol\left( \frac{S[k+1]}{k+1} \right),
\]
 and that $\dim(S)=d$. Then, from inequality (\ref{eq:equality}) we deduce that 
\[
\sum_{j=1}^m \left( \vol \left( \frac{M_j +B_j}{k+1}  \right) - \vol \left( \frac{M_j }{k} \right) \right) \leq \varepsilon.
\]
For $j=1,2,\ldots,m$, set $\delta_j = \vol \left( \frac{1}{k+1} (M_j +B_j) \right) - \vol \left( \frac{1}{k} M_j \right)$.
Then, clearly $\sum \delta_j \leq  \varepsilon$. On the other hand, by Lemma~\ref{lem:main}, 
for every $j=1,2,\ldots,m$, we have $\vol(k \conv B_j ) - \vol(M_j) \leq C(k,d) \delta_j$,  where $C(k,d)$ is defined in (\ref{eq:ckd}).
Thus, summing on $j$, it follows that
$$
 \varepsilon C(k,d)\geq \vol(\conv (kA)) - \vol (S[k] \cap \conv(kA)),
$$ 
implying that
$\varepsilon \left( k^d + C(k,d) \right) \geq \vol(\conv(kS)) - \vol(S[k])$.
This inequality holds for any value $\varepsilon > 0$, and hence, $\vol(\conv(S)) = \vol\left( \frac{1}{k} S[k] \right)$, or equivalently,
$\vol \left( \conv(S) \setminus \frac{1}{k} S[k] \right) = 0$. Since $\conv(S)$ is a compact, convex set with nonempty interior, and $\frac{1}{k} S[k]$ is compact, to show the equality $\conv(S) = \frac{1}{k} S[k]$, we may apply the argument at the end of the proof of Lemma~\ref{lem:main}.
\end{proof}

\section{Conjecture~\ref{conj:BMW} for connected sets}\label{sec3}

In the first few lemmata we collect some elementary properties of the Minkowski sum of connected sets.
Throughout this section, $e_1, e_2$ denotes the elements of the standard orthonormal basis of $\Re^2$.

\begin{lem}\label{lem:boundary}
Let $A \subset \Re^d$ be a compact set with a connected boundary and let  $\bd A \subseteq B \subseteq A$. Then $B + B=A+A$.
\end{lem}

\begin{proof}
We have $\bd A + \bd A \subseteq B+B \subseteq A+A$. Thus it is sufficient to prove that $\bd A + \bd A=A+A$.
Clearly, $A + A \supseteq \bd A + \bd A$. We show that $\frac{A + A}{2} \subseteq \frac{\bd A + \bd A}{2}$, which then yields the assertion.
Consider a point $p \in \frac{A + A}{2}$. Then $p$ is the midpoint of a segment whose endpoints are points of $A$.
Let $\chi_p : \Re^d \to \Re^d$ be the reflection about $p$  defined by $\chi_p(x)=2p-x$, for $x\in\Re^d$. To prove that $p \in \frac{\bd A + \bd A}{2}$ we need to show that for some $q \in \bd A$, we have $\chi_p(q) \in \bd A$. To do this, let us define $f_p(x)$ ($x \in \Re^d$) as the signed distance of $\chi_p(x)$ from the boundary of $A$, where the sign is positive if $\chi_p(x) \notin A$, and not positive if $\chi_p(x) \in A$. Here we remark that since $A$ is compact, $\bd A$ is compact as well.
Let $x_1$ be a point of $\bd A$ farthest from $p$. If $\chi_p(x_1) \in A$ then $\chi_p(x_1) \in \bd A$, and we are done. Thus, assume that $\chi_p(x_1) \notin A$, implying that $f_p(x_1) > 0$.
Now, since $p \in \frac{A + A}{2}$, we have some $y \in A$ such that $\chi_p(y) \in A$. Let $L$ be the line through $y$, $p$ and $\chi_p(y)$. Let $y'$ and $y''$ be points of $L \cap \bd A$ closest to $y$ and $\chi_p(y)$, respectively.  Then the segments $[y,y']$ and $[\chi_p(y),y'']$ are included in $A$. If $0 < |y'-y| \leq |y''-\chi_p(y)|$, then $y' \in \bd A$ and $\chi_p(y') \in  [\chi_p(y),y'']\subset  A$. If $0 < |y''-\chi_p(y)| \leq |y'-y|$, then the same holds for $y''$ in place of $y'$. Thus, it follows that for some point $x_2 \in \bd A$, $\chi_p(x_2) \in A$. If $\chi_p(x_2) \in \bd A$, then we are done, and so we may assume that $\chi_p(x_2) \in \inter A$, which yields that $f_p(x_2) < 0$.

We have shown that $f_p : \bd A \to \Re$ attains both a positive and a negative value on its domain. On the other hand, since $f$ is continuous and $\bd A$ is connected, $f_p(q) = 0$ for some $q \in \bd A$, from which the assertion readily follows.
\end{proof}

\begin{rem}
Lemma~\ref{lem:boundary} holds also for the boundary of the external connected component of $\Re^d \setminus A$ in place of $\bd A$.
\end{rem}

\begin{rem}
We note that the equality $A_1 + A_2 = \bd A_1 + \bd A_2$ does not hold in general for different compact sets $A_1, A_2$ with connected boundaries. To show it, one may consider the sets $A_1 = B_2^2$ and $A_2 = \varepsilon B_2^2$ for some sufficiently small value of $\varepsilon$, where  $B_2^d$ be the Euclidean unit ball of dimension $d$ centered at the origin.

\end{rem}

\begin{rem}
Lemma~\ref{lem:boundary} does not hold if we omit the condition that $\bd A$ is connected. To show it, we may choose $A$ as the union of $B_2^2$ and a singleton $\{ p \}$ with $|p|$ being sufficiently large.
\end{rem}

\begin{cor}\label{cor:ksumboundary}
If $A$ is a compact set with a connected boundary then $A + A = A + \bd A = \bd A + \bd A$. Thus, for any positive integer $k \geq 2$, we have  $A[k] = \bd A[k]$.
\end{cor}

\begin{cor}
 Let $d\ge2$ and $k \geq  \max\{ 2, (d-1)(d-2) \}$. Let $A$ be a compact set such that $\partial S \subseteq A \subseteq S$ for some compact, star-shaped set $S \subset \Re^d$. Then we have
\[
\vol\left( \frac{1}{k}A[k]\right) \leq \vol\left( \frac{1}{k+1}A[k+1]\right).
\]          
\end{cor}

\begin{proof} 
Without loss of generality, we may assume that $S$ is star-shaped with respect to the origin.
Set $S'=S + \varepsilon B_2^d$ for some small value $\varepsilon > 0$.

First, we show that $\partial S'$ is path-connected. Let $L$ be a ray starting at $o$. Since $o \in \inter S'$, $L \cap \bd S' \neq \emptyset$.
Let $p \in L \cap  \bd S'$. Then there is a point $q \in S$ such that $|q-p| = \varepsilon$. Now, if $x$ is any relative interior point of $[o,q]$, then the line through $x$ and parallel to $[p,q]$ intersects $[o,q]$ at a point at distance less than $\varepsilon$ from $x$. Since $[o,q] \subseteq S$, from this it follows that $x \in S + \varepsilon \inter B_2^d \subseteq \inter S'$. In other words, for any $p \in \bd S'$, all points of $[o,p]$ but $p$ lie in $\inter S'$. Thus, $L \cap \bd S'$ is a singleton for any ray $L$ starting at $o$.

Let $0 < r < R$ such that $\bd S' \subset H = R B_2^d \setminus (r \inter B_2^d$).
Let $P : H \to \Sph^{d-1}$ be the central projection to $\Sph^{d-1}$. Note that $P$ is Lipschitz, and thus continuous on $H$, and its restriction $P|_{\bd S'}$ to $\bd S'$ is bijective. On the other hand, since $\bd S'$ (as also $S'$) are compact, this implies that the inverse of $P|_{\bd S'}$ is continuous, that is, $\bd S'$ and $\Sph^{d-1}$ are homeomorphic. Thus, $\bd S'$ is path-connected.

On the other hand, $\bd S \subseteq A \subseteq S$ implies that $A'=A + \varepsilon B_2^d \subseteq S'$, and $\bd S' \subseteq \bd S + \varepsilon \Sph^{d-1} \subseteq \bd S + \varepsilon B_2^d \subseteq A'$. Now, we may apply Lemma~\ref{lem:boundary} and Corollary~\ref{cor:ksumboundary}, and obtain that for any value of  $k \geq 2$, $A'[k]=S'[k]$.
Thus, by Theorem~\ref{thm:main} it follows that 
\[
\vol\left(\frac{A[k]}{k}  + \varepsilon B_2^d \right)=\vol\left(\frac{A'[k]}{k} \right) \leq \vol\left(\frac{A'[k+1]}{k+1}  \right) =\vol\left(\frac{A[k+1]}{k+1}  + \varepsilon B_2^d \right).
\] 
 On the other hand, for any compact set $C$ the function $t\mapsto \vol(C+tB_2^d)$ is continuous on $[0,+\infty)$, see for example \cite{FM}, hence  $\lim_{\varepsilon \to 0^+} \vol\left(\frac{1}{m} A[m] + \varepsilon B_2^d \right) = \vol\left(\frac{1}{m} A[m] \right)$, for any integer $m$ which implies the corollary.
\end{proof}

Let us denote the closure of a set $A \subset \Re^d$ by $\cl(A)$.

\begin{prop}\label{lem:simple_curve}
Let $\gamma \subset \Re^2$ be a simple continuous curve connecting $o$ and $e_1$ such that its intersection with the $x$-axis is $\{ o,e_1\}$.
Let $D$ be the interior of the closed Jordan curve $\gamma \cup [o,e_1]$. For $i=0,1$, let $\gamma_i = \frac{i}{2} e_1 + \frac{1}{2} \gamma$, and $D_i = \frac{i}{2} e_1 + \frac{1}{2} D$.
Then $\cl \left(D \setminus (D_0 \Delta D_1) \right) \subseteq \frac{1}{2} \gamma[2]$, where $\Delta$ denotes symmetric difference.
\end{prop}

\begin{proof}
For convenience, we assume that $\gamma$ lies in the half plane $\{ y \leq 0 \}$.
As in the proof of Lemma~\ref{lem:boundary}, let $\chi_p : \Re^2 \to \Re^2$ denote the reflection about $p \in \Re^2$   defined by $\chi_p(x)=2p-x$, and note that $p \in \frac{1}{2} \gamma[2]$ if and only if there is some point $q \in \gamma$ such that $\chi_p(q) \in \gamma$, or in other words, if $\gamma \cap \chi_p(\gamma) \neq \emptyset$. Let $L$ denote the $x$-axis, $L_p=\chi_p(L)$, and let $S$ be the infinite strip between $L$ and $L_p$ (cf. Figure~\ref{fig:cont_curve}).

\begin{figure}[ht]
\begin{center}
\includegraphics[width=.6\textwidth]{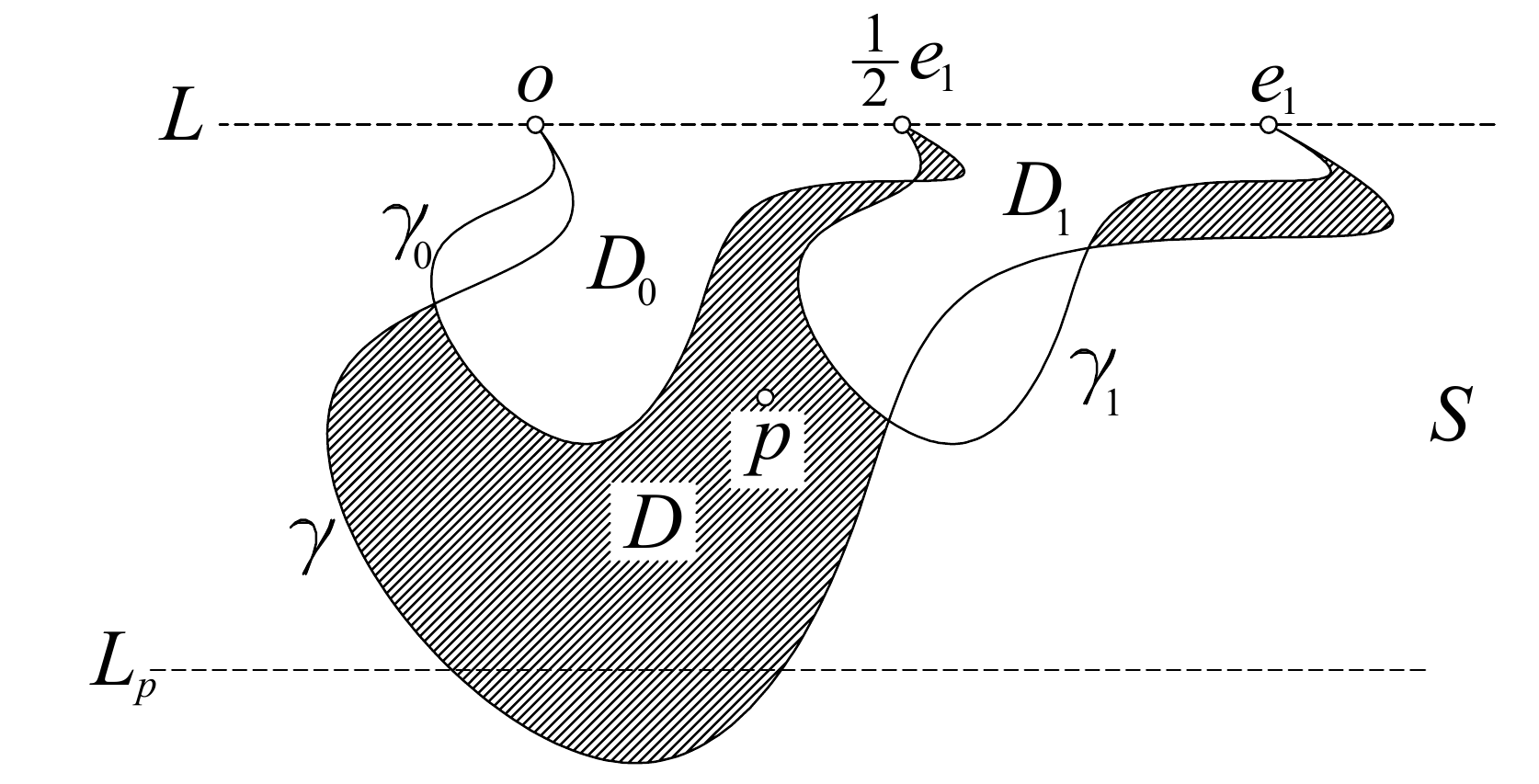}
\caption{An illustration for Proposition~\ref{lem:simple_curve}. The dashed region belongs to $\frac{1}{2} \gamma[2]$.}
\label{fig:cont_curve}
\end{center}
\end{figure}

First, observe that $o,e_1 \in \gamma$ yields that $\gamma_0 \cup \gamma_1 \subset \frac{1}{2} \gamma[2]$, and $\gamma \subset \frac{1}{2} \gamma[2]$ trivially holds. Thus, we need to show that if for some point $p$ we have $p \in D \setminus \cl (D_0 \cup D_1)$ or $p \in D_0 \cap D_1 \cap D$, then $p \in \frac{1}{2}\gamma[2]$. We do it only for the case $p \in D \setminus \cl (D_0 \cup D_1)$ since for the second case a similar argument can be applied.

Consider some point $p \in D \setminus (D_0 \cup D_1)$. Then $p \notin \cl (D_0 \cup D_1)$ yields that $\chi_p(o) = 2p \notin \cl D$, and the relation $\chi_p(e_1) \notin \cl D$ follows similarly.

\emph{Case 1}: $\gamma \subset S$. Note that in this case $\chi_p(\gamma) \subset S$.
Since $p \in D$ and $\chi_p(o) \notin \cl D$, $\bd D = \gamma \cup [o,e_1]$ and $[\chi_p(o),p] \cap [o,e_1] = \emptyset$, it follows by the continuity of $\gamma$ that $\gamma \cap [\chi_p(o),p] \neq \emptyset$. Hence, by the compactness of $\gamma$, there is a point $x \in \gamma \cap [\chi_p(o),p]$ closest to $p$.
By its choice, $\chi_p(x) \in D \cup \gamma$. If $\chi_p(x) \in \gamma$, we are done, and thus, we assume that $\chi_p(x) \in D$.
This implies that $\chi_p(\gamma)$ contains both interior and exterior points of $D$. On the other hand, since $\chi_p(\gamma) \subset S$, this implies that $\chi_p(\gamma) \cap \gamma \neq \emptyset$.

\emph{Case 2}: $\gamma \not\subset S$.
Let $\gamma_p = \gamma \cap  S$, and let  $\bar{\gamma}_0$ and $\bar{\gamma}_1$ denote the connected components of $\gamma_p$ containing $o$ and $e_1$, respectively.
For $i=0,1$, we denote the endpoint of  $\bar{\gamma}_i$ on $L_p$ by $x_i$. Clearly, since $\gamma$ is simple and continuous,  $x_0$ is on the left-hand side of  $x_1$, and the curve  $\bar{\gamma}_0 \cup [x_0,x_1] \cup \bar{\gamma}_1 \cup [o,e_1]$ is a Jordan curve. We denote the interior of this curve by $D_p$.

Consider the case where $p \notin D_p$. Then $p$ is an exterior point of $D_p$, and there is a connected component $\gamma^*$ of $\gamma_p$, with endpoints on $L_p$, that separates $p$ from $L$. Since the reflections of the endpoints of $\gamma^*$ about $p$ lie on $L$, we may apply the argument in Case 1, and obtain that $\emptyset \neq \gamma^* \cap \chi_p(\gamma^*) \subseteq \gamma \cap \chi_p(\gamma)$.
Thus, we may assume that $p \in D_p$.

If $\chi_p( x_0) \in [o,e_1]$, then the continuity of  $\bar{\gamma}_0$ and $\chi_p(o) \notin \cl D$ implies that $\emptyset \neq  \gamma \cap \chi_p(\bar{\gamma}_0) \subseteq \gamma \cap \chi_p(\gamma)$. If $\chi_p( x_1 ) \in [o,e_1]$, then we may apply a similar argument, and thus we may assume that $ \chi_p(x_0), \chi_p(x_1) \notin [o,e_1]$. This implies that either $[o,p_1] \subset  [\chi_p(x_0), \chi_p(x_1)]$, or  $[\chi_p(x_0), \chi_p(x_1)]$ and $[o,p_1]$ are disjoint.

 The relation $[o,p_1] \subset [\chi_p(x_0), \chi_p(x_1)]$ yields $[\chi_p(o),\chi_p(p_1)] \subset [x_0, x_1]$, and, by the previous argument, we have $\emptyset \neq \chi_p(\bar{\gamma_0}) \cap \bar{\gamma}_1 \subseteq \gamma \cap \chi_p(\gamma)$. Thus, we are left with the case where $[\chi_p(x_1), \chi_p(x_2)]$ and $[o,p_1]$ are disjoint; without loss of generality we may assume that $\chi_p(x_1)$,  $\chi_p(x_0)$, $o$ and $e_1$ are in this consecutive order on $L$. Let $U$ be the closure of the connected component of  $S \setminus \bar{\gamma}_0$ containing $\bar{\gamma}_1$.
Then $\chi_p(p) = p \in \inter U \cap \chi_p(U)$, implying that $\emptyset \neq \gamma_1 \cap \chi_p(\gamma_1) \subseteq \gamma \cap \chi_p(\gamma)$.
\end{proof}

The proof of Lemma~\ref{lem:convex_curve} is based on the idea of the proof of Proposition~\ref{lem:simple_curve}, with some necessary modifications.

\begin{lem}\label{lem:convex_curve}
Let $k \geq 2$, and let $\gamma \subset \Re^2$ be a convex, continuous curve connecting $o$ and $e_1$ such that its intersection with the $x$-axis is $\{ o,e_1\}$.
Let $D$ be the interior of the closed Jordan curve $\gamma \cup [o,e_1]$. For $i=0,1,\ldots,k-1$, let $\gamma_i = \frac{i}{k} e_1 + \frac{1}{k} \gamma$, and  $D_i = \frac{i}{k} e_1 + \frac{1}{k} D$. 
Then $\cl \left( D \setminus (\bigcup_{i=1}^k D_i) \right) \subseteq \frac{1}{k} \gamma[k]$, and for any $i \neq j$, $D_i \cap D_j  \subseteq \frac{1}{k}\gamma[k]$.
\end{lem}

\begin{proof}
First observe that $D$ is convex, hence $D_i$ is contained in $D$ for all values of $i$.
Let us denote the $x$-axis by  and, for any $p \in \Re^2$, let $\chi_p^k:\Re^2 \to \Re^2$ be the homothety with center $p$ and ratio $-\frac{1}{k-1}$  defined by $\chi_p^k(x)=\frac{k}{k-1}p-\frac{x}{k-1}$, for $x\in\Re^2$. Furthermore, we set $L_p^k = \chi_p^k(L)$, and denote the infinite strip between $L$ and $L_p^k$ by $S$.
The assertion for $k=2$ is a special case of Proposition~\ref{lem:simple_curve}. To prove it for $k \geq 3$, we apply induction on $k$, and assume that the lemma holds for $\gamma[k-1]$.

Let $p \in \cl\left( D \setminus (\bigcup_{i=1}^k D_i) \right)$.
Clearly, since $(\bd D) \setminus (\bigcup_{i=1}^k D_i) = \gamma \subseteq \gamma[k]$, we may assume that $p \in D$.
By the induction hypothesis for $\frac{k-1}{k} \gamma$ , if $p \in X_1=\frac{k-1}{k} \cl D$, then $p \in \frac{k-1}{k} \cdot \frac{1}{k-1} \gamma[k-1] = \frac{1}{k} \gamma[k-1] \subseteq \frac{1}{k} \gamma[k]$. Similarly, if $p \in X_2= \frac{1}{k}e_1 + \frac{k-1}{k} \cl D$, then $p \in \frac{1}{k}e_1 + \frac{1}{k} \gamma[k-1] \subseteq \frac{1}{k} \gamma[k]$. Thus, assume that $p \notin X_1 \cup X_2$, which yields that $\chi_p^k(o)$ and $\chi_p^k(e_1)$ are in the exterior of $D$.
Let the (unique) intersection point of $[p,\chi_p^k(o)]$ and $\gamma$ be $q_1$ and the (unique) intersection point of $[p,\chi_p^k(e_1)]$ and $\gamma$ be $q_2$.
As $\chi_p^k(q_1) \in [o,p]$, the convexity of $D$ implies that $\chi_p^k(q_1) \in D$, and the containment $\chi_p^k(q_2) \in D$ follows similarly.

Similarly like in Proposition~\ref{lem:simple_curve}, if $\gamma \subset S$, then by continuity, $\gamma \cap \chi_p^k(\gamma) \neq \emptyset$, which implies the containment $p \in \frac{1}{k} \gamma[k]$.
Assume that $\gamma \not\subset S$. Then $S \cap \gamma$ has two connected components $\gamma_1$, $\gamma_2$, where we choose the indices such that $o \in \gamma_1$, and $e_1 \in \gamma_2$. Clearly, we have either $q_1 \in \gamma_2$, $q_2 \in \gamma_1$, or both. If $q_1 \in \gamma_2$, then the containment relations
$\chi(q_1) \in D$, $\chi(e_1) \notin \cl D$, and $\chi_p^k(\gamma_2) \subset S$ yield that $\emptyset \neq \gamma_1 \cap \chi_p^k(\gamma_2) \subset \gamma \cap \chi_p^k(\gamma)$. If $q_2 \in \gamma_1$, then the assertion follows by a similar argument.

Finally, we consider the case that $p \in D_i \cap D_j$ for some $i < j$. In this case the convexity of $D$ implies that $p \in D_s$ for any $i \leq s \leq j$.
This yields that there are some distinct values $i,j \leq k-1$ or $i,j\geq 2$ such that $p \in D_i \cap D_j$. Thus, the assertion readily follows from the induction hypothesis.
\end{proof}

Lemma~\ref{lem:convexholes} is a variant of Lemma~\ref{lem:main} for some path-connected sets in $\Re^2$.

\begin{lem}\label{lem:convexholes}
Let $k \geq 2$ and $\gamma$ be a bounded convex curve in $\Re^2$, and
let $\gamma[k] \subseteq M \subseteq k \conv \gamma$. Then
\[
\area \left( \frac{1}{k} M \right) \leq \area \left( \frac{1}{k+1} (M+\gamma) \right).
\]
\end{lem}

\begin{proof}
If $\gamma$ is closed, then Lemma~\ref{lem:boundary} yields that $\frac{1}{k} \gamma[k]= \conv \gamma$ for all $k \geq 2$, which clearly implies the statement.
Assume that $\gamma$ is not closed. Since the   inequalities in Lemma~\ref{lem:convexholes} do not change under affine transformations, we may assume that
the endpoints of $\gamma$ are $o$ and $e_1$, and the $x$-axis is a supporting line of $\conv \gamma$.

\noindent Let us define 
\[
D = \conv \gamma,  \alpha = \area(D \cap (e_1+D)), \mbox{    and    }\beta = \area \left( D \cap \left((e_1+D) \cup (-e_1+D) \right) \right).
\]
Note that $0 \leq \alpha \leq \beta \leq  2\alpha$.
Let $D_i = i e_1 + D$ for $i=0,1,\ldots,k$.
For $0 \leq i \leq k-1$, let $\mu_i$ be the area of the region of $M$ in $D_i$ that do not belong to any $D_j$, $j \neq i$, where we note that
since $k \geq 2$, by Lemma~\ref{lem:convex_curve} we have that all other points of $D_i$ belong to $M$.
Similarly, for $0 \leq i \leq k$, let $\lambda_i$ be the area of the region of $M+\gamma$ in $D_i$ that do not belong to any $D_j$, $j \neq i$.
An elementary computation shows that
\begin{equation}\label{eq:areaM}
\begin{split}
\area(M) = & k^2 \area(D) - 2(\area(D)-\alpha) - (k-2) (\area(D)-\beta) + \sum_{i=0}^{k-1} \mu_i\\ = &(k^2-k) \area(D) + 2\alpha +(k-2)\beta + \sum_{i=0}^{k-1} \mu_i,
\end{split}
\end{equation}
and similarly,
\begin{equation}\label{eq:area_Mgamma}
\area(M+\gamma) = (k^2+k) \area(D) + 2\alpha + (k-1) \beta + \sum_{i=0}^{k} \lambda_i.
\end{equation}
Since $o, e_1 \in \gamma$, we have $M, e_1 + M \subseteq M+\gamma$. Thus, $\lambda_0 \geq \mu_0$, $\lambda_k \geq \mu_{k-1}$,
$\lambda_1 \geq \max \{ \mu_0-(\beta-\alpha), \mu_1 \}$, $\lambda_{k-1} \geq \max \{ \mu_{k-2}, \mu_{k-1}-(\beta-\alpha) \}$, and
for $2 \leq i \leq k-2$, $\lambda_i \geq \max \{ \mu_{i-1},\mu_i\}$.
Since $\lambda_i \geq \frac{i}{k} \mu_{i-1} + \frac{k-i}{k}\mu_i$ if $2 \leq i \leq k-2$, and
$\lambda_i \geq \frac{i}{k} \mu_{i-1} + \frac{k-i}{k}\mu_i - \frac{1}{k}(\beta-\alpha)$ if $i=1$ or $i=k-1$, it follows that
\[
\sum_{i=0}^{k} \lambda_i  \geq \frac{k+1}{k}\sum_{i=1}^{k-1} \mu_i - \frac{2}{k} (\beta-\alpha).
\]
Thus, by (\ref{eq:areaM}),
\[
\sum_{i=0}^{k} \lambda_i \geq \frac{k+1}{k} \left( \area(M) - (k^2-k) \area(D) - 2\alpha - (k-2)\beta \right) - \frac{2}{k} (\beta-\alpha).
\]
After substituting this into (\ref{eq:area_Mgamma}) and simplifying, we obtain
\[
\area(M+\gamma) \geq \frac{k+1}{k} \area(M) + (k+1) \area(D),
\]
which yields
\[
\area\left(\frac{1}{k+1} (M + \gamma)\right) \geq \frac{k}{k+1} \area \left( \frac{1}{k} M \right) + \frac{1}{k+1} \area(D).
\]
Thus, the inequality $\area \left( \frac{1}{k} M \right) \leq \area(D)$ yields the assertion.
\end{proof}

In Theorem~\ref{thm:convexholes}, by an open topological disc we mean the bounded connected component defined by a Jordan curve,  and recall that a convex body is a compact, convex set with nonempty interior.

\begin{thm}\label{thm:convexholes} 
Let $k \geq 2$. Let $K$ be a plane convex body, and let $\mathcal{F} = \{ F_i : i \in I\}$ be a family of pairwise disjoint  open topological discs such that if $F_i \cap \bd K\neq\emptyset$ then $F_i \cap \bd K$ is a connected curve and $F_i$ is convex.
Let $X = K \setminus \left( \bigcup_{i\in I} F_i \right)$.
Then
\[
\area\left( \frac{1}{k} X[k] \right) \leq \area\left( \frac{1}{k+1} X[k+1] \right).
\]
\end{thm}

\begin{proof}
 Clearly, we may assume that each $F_i$ intersects $K$, and also for each $F_i$, $(\bd K) \setminus F_i$ is infinite, since removing the first type discs does not change $X$, and if there is some $F_i$ such that $(\bd K) \setminus F_i$ is finite, then $X$ is either $\emptyset$ or a singleton, and in both cases the statement is trivial. Thus, we have that if $F_i$ intersects $\bd K$, then the boundary of the convex set $F_i \cap K$ consists of the two connected, convex curves $F_i \cap \bd K$ and $K \cap \bd F_i$.

First, note that since each member of $\mathcal{F}$ has positive area, it has countably many elements; indeed, for any $\delta > 0$ there are only finitely many elements $F_i$ of $\mathcal{F}$ for which $\area(F_i \cap K) \geq \delta$, and thus, we may list the elements according to area.
Furthermore, since $X$ is compact, $\area(X)$ exists.

By Lemma~\ref{lem:boundary}, we may assume that every member of $\mathcal{F}$ intersects $\bd K$;  indeed, if some $F_i$ does not intersect $\bd K$, then $\bd F_i$ is a compact, connected set in $X$, implying that $F_i \subseteq \frac{1}{k} (\bd F_i)[k] \subseteq \frac{1}{k} X[k]$ for all $k \geq 2$.
For any $i \in I$, let $\gamma_i$ denote the part of $\bd F_i$ in $K$. Clearly, $\gamma_i$ is a convex curve, and  the segment connecting its endpoints lie in $K$ by convexity. As the two endpoints of $\gamma_i$ are in $\bd K$, the line through them supports $K \setminus F_i$.
Choose some finite subfamily $I_{\varepsilon} \subseteq I$ such that $\area \left( X_{\varepsilon} \setminus X\right) \leq \varepsilon$, where $X_{\varepsilon} = K \setminus \left( \bigcup_{i \in I_{\varepsilon}}  F_i \right)$.
This is possible, since for any ordering of the elements, $\sum_{i \in I} \area(K \cap F_i)$ is a bounded series with positive elements, and hence, it is absolute convergent, and convex sets with small area and bounded diameter are contained in a small neighborhood of their boundary.

\begin{figure}[ht]
\begin{center}
\includegraphics[width=.3\textwidth]{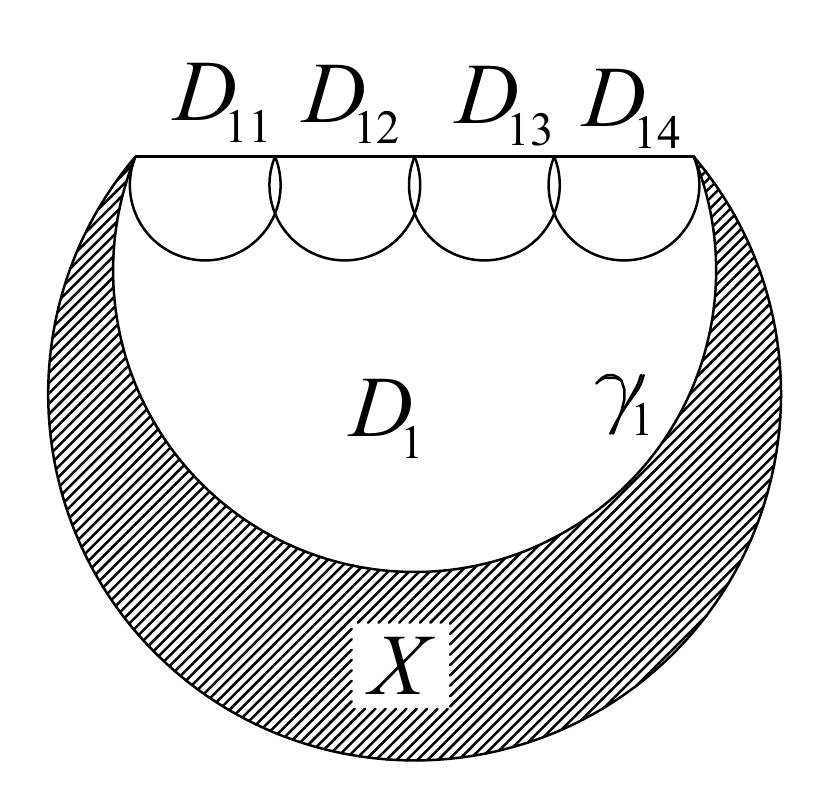}
\caption{An illustration for the proof of Theorem~\ref{thm:convexholes}.}
\label{fig:thm2}
\end{center}
\end{figure}

For any $i \in I_{\varepsilon}$, we set $D_i = F_i \cap K$, and observe that $D_i$ is a convex set separated from $X_{\varepsilon}$ by the convex curve $\gamma_i$.
 Let the endpoints of $\gamma_i$ be $q_i^1$ and $q_i^2$, and let $D_{i1}$ be the homothetic copy of $D_i$ with ratio $\frac{1}{k}$ and center $\_i^1$. Furthermore, for $j=2,3,\ldots,k$, let $D_{ij}= \frac{j-1}{k} \left( q_i^2 - q_i^1 \right) + D_{i1}$ (cf. Figure~\ref{fig:thm2}). Then, by Lemma~\ref{lem:convex_curve}, $\frac{1}{k} \gamma_i[k] \subseteq \frac{1}{k} X_{\varepsilon}[k]$ contains all points of $D_i$ belonging to none of the $D_{ij}$s or to at least two of them.
Let $M_i = (X[k] \cap (kD_i))$. Then $M_i \subseteq \conv(kD_i)$, and thus, Lemma~\ref{lem:convexholes} yields that
\[
\area \left( \frac{1}{k} M_i \right) \leq \area \left( \frac{1}{k+1} (M_i+\gamma_i) \right).
\]
On the other hand, with the notation $D_{\varepsilon} = \bigcup_{i \in I_{\varepsilon}} D_i$, we have
\[
\area \left( \frac{1}{k}  X[k] \cap D_{\varepsilon} \right) = \sum_{i \in I_{\varepsilon}} \area \left( \frac{1}{k} M_i \right),
\]
and
\[
\area \left(  \frac{1}{k+1} X[k+1] \cap D_{\varepsilon} \right) \geq \sum_{i \in I_{\varepsilon}} \area \left( \frac{1}{k+1} (M_i+\gamma_i) \right),
\]
and thus, we have $\area \left( \frac{1}{k}  X[k] \cap D_{\varepsilon} \right) \leq \area \left( \frac{1}{k+1} X[k+1] \cap  D_{\varepsilon}\right)$.
On the other hand, since $\area(X_{\varepsilon} \setminus X) < \varepsilon$, $X_{\varepsilon} \cup  D_{\varepsilon} = \conv X$, and $X \subseteq X_{\varepsilon}$,
we have that $\area(\frac{1}{m} X_[m] \setminus  D_{\varepsilon} ) \leq \varepsilon$ for all $m \geq 1$. This implies that
\[
\area \left( \frac{1}{k} X_[k] \right) \leq \area \left( \frac{1}{k+1} X[k+1] \right) - \varepsilon.
\]
 This holds for all $\varepsilon >0$, which yields the assertion.
\end{proof}

\section{Additional remarks and questions}\label{sec:remarks}

\begin{rem}
One can ask if the statement of Theorem~\ref{thm:main} holds for arbitrary measure instead of volume. The answer to this question is negative.
Indeed, consider the measure $\mu(K) = \vol (K \cap C)$, where $C=\left[ - \frac{1}{d} , \frac{1}{d} \right]^d$ and $S=\bigcup_{i=1}^d [o,e_i]$, where $e_1,e_2,\ldots,e_d$ are the vectors of the standard orthonormal basis. Then, clearly, we have
\[
\mu\left(\frac{1}{2k} S[2k]\right) = \frac{1}{2^d} \vol(C) > \mu\left(\frac{1}{2k+1} S[2k+1]\right).
\]
\end{rem}

\begin{rem}
The statement of Theorem~\ref{thm:main} does not hold for arbitrary measures even for rotationally invariant measures in the plane:
for any value of $k$ there is a compact, star-shaped set $S \subset \Re^2$ such that
$\vol\left(\frac{1}{k}S[k] \cap B_2^2 \right) > \vol\left(\frac{1}{k+1}S[k+1] \cap B_2^2 \right)$.
To prove this, set $S=[o,e_1] \cup [o,e_2]$, and let $E$ denote the ellipse centered at $o$ and containing the points $(1-1/k,0)$ and $(1-2/k,1/k)$. It is an elementary computation to check that in this case $\vol\left(\frac{1}{k}S[k] \cap E \right) = \frac{1}{4} \vol(E)$. On the other hand,
the boundary point $(1-2/(k+1),1/(k+1))$ of $\frac{1}{k+1}S[k+1]$ lies in $\inter (E)$, which implies that $\vol\left(\frac{1}{k+1}S[k+1] \cap E \right) < \frac{1}{4} \vol(E)$. Now, if $f : \Re^2 \to \Re^2$ is defined as the linear transformation mapping $E$ into $B_2^2$, then $f(S)$ satisfies the required conditions.
\end{rem}

One can use star-shaped sets together with ideas from \cite{FMMZ2} to give a negative answer to a more general version of Conjecture~\ref{conj:BMW}, also from \cite{BMW}.

\begin{conj}[Bobkov-Madiman-Wang]\label{conj:BMW2} 
For any $k \geq 2$, and any  compact sets $A_1, A_2, \ldots, A_{k+1}$ in $\Re^d$,  we have
\[
\vol\left(\sum_{i=1}^{k+1} A_i\right)^{1/d}  \geq \frac{1}{k} \sum_{i=1}^{k+1} \vol\left( \sum_{j \neq i} A_j\right)^{1/d}.
\]
in particular,  for $k=2$,
\begin{equation}\label{eqcon3}
\begin{split}
\vol(A_1+ A_2+& A_3)^{1/d}  \\ \geq & \frac{1}{2}  \left(\vol\left(A_1+A_2\right)^{1/d}   +  \vol\left(A_1+A_3\right)^{1/d}  + \vol\left(A_2+A_3\right)^{1/d}   \right).
\end{split}
\end{equation}
\end{conj}

The above  conjecture is trivial for convex sets. Moreover, (\ref{eqcon3}) is true  when $A_1=A_2$ and $A_1$ is convex. Indeed, in this case (\ref{eqcon3}) is equivalent to
\[
\vol\left(A_1+A_1+A_3\right)^{1/d} \geq       \vol\left(A_1\right)^{1/d}   +  \vol\left(A_1+A_3\right)^{1/d},
\]
 which follows from the Brunn-Minkowski inequality \cite{Sch}. 
 
It was proved in \cite{FMMZ2} that Conjecture \ref{conj:BMW2} is true in $\Re$. Since an affirmative answer to Conjecture~\ref{conj:BMW2} implies also Conjecture~\ref{conj:BMW}, the former is also false for $d \geq 12$ by \cite{FMMZ1, FMMZ2}. Here we show that  Conjecture~\ref{conj:BMW2} is false in $\Re^d$  for $d \geq 7$.

\begin{prop}\label{rem:counterexample}
For any $d \geq 7$, there are compact, star-shaped sets $A_1, A_2, A_3 \subset \Re^d$ satisfying
\[
\vol\left(A_1+ A_2+ A_3\right)^{1/d} < \frac{1}{2}  \left(        \vol\left(A_1+A_2\right)^{1/d}   +  \vol\left(A_1+A_3\right)^{1/d}  + \vol\left(A_2+A_3\right)^{1/d}   \right).
\]
\end{prop}

\begin{proof}
We give the proof for $d=7$ and the result follows for $d>7$ by taking direct products with a cube. 
 Consider the sets 
\[
A_1=[0,1]^4\times \{0\}^3; A_2=\{0\}^4 \times [0,1]^3 
 \mbox{ and  } A_3=([0,a]^4 \times \{0\}^3) \cup (\{0\}^4 \times [0,b]^3),
\]
where we select $a,b > 0$ later.  Since these sets are lower dimensional, one has $\vol(A_1)=\vol(A_2)=\vol(A_3)=0$.
An elementary consideration shows that
\[
\vol(A_1+A_3)= b^3,    \vol(A_2+A_3)=a^4 \mbox{ and } \vol(A_1+A_2)=1,
\]
and
\[
\vol(A_1+A_2+A_3)=(a+1)^4+(b+1)^3 -1.
\]
The last step is to show that, with $a=3$ and $b=6$, the quantity
\[
((a+1)^4+(b+1)^3 -1)^{1/7} -\frac{1}{2}\left(a^{4/7}+b^{3/7}+1 \right) 
\]
is negative, which gives a counterexample to (\ref{eqcon3}).
\end{proof}

\end{document}